\documentclass[12pt]{amsart}
	\usepackage{amsmath,amsfonts,amsthm} 
  \usepackage{amssymb,latexsym}  
  \usepackage[latin1]{inputenc}   
	\usepackage{hyperref}      
  \usepackage{fancyhdr}    
	\usepackage{url}

	\newtheorem{theorem}{Theorem}
	\newtheorem{lemma}{Lemma}
	\newtheorem{corollary}{Corollary}
\title{  $p$-adic Zeros of  Quintic Forms} 
\author{ Jan H. Dumke }
\begin{document}
\begin{abstract} 
It is shown that a quintic form  over a $p$-adic field with at least $26$ variables has a non-trivial zero,  providing that the cardinality of the residue class field exceeds $9$. 
\end{abstract}
\maketitle	
\let\thefootnote\relax
\let\thefootnote\relax 

\section{Introduction}
Let $F(x_1,\dots,x_n)$ denote a form of degree $d$ over a $p$-adic field $\mathbb{K}$. It is a conjecture of E. Artin from the 1930s, that $F$ has a non-trivial zero as soon as $n>d^2$. Although this is known to be false
for many $d$ (for instance, see \cite{MR0197450} for a $2$-adic quartic form) the conjecture has been partially verified by Ax and Kochen \cite{MR0184930}. They showed that for every $d$ there exists a positive integer $q_0(d)$, such that Artin's conjecture holds whenever the cardinality $q$ of the residue class field  
exceeds $q_0(d)$. However, little is  known about the actual values of $q_0(d)$. Brown \cite{MR494980} has given a huge, but explicit bound on $q_0(d)$. If we write $a\uparrow b$ for $a^b$ it can be stated as
\begin{align}
q_0(d)\leq 2\uparrow (2\uparrow (2\uparrow (2\uparrow (2\uparrow (d\uparrow (11\uparrow (4d))))))).		\notag
\end{align}
If $d$ is neither composite nor a sum of composite numbers, better bounds are available. Besides the classical result $q_0(2)=1$ (Hasse \cite{Hasse}) and $q_0(3)=1$ (Lewis \cite{Lewis1}) this concerns in fact $d=5,7,11$ only. 
Leep and Yeomans \cite{MR1382749} have shown $q_0(5)\leq  43$ and later this has been improved by Heath-Brown \cite{MR2595750}. He proved that a quintic form over $\mathbb{Q}_p$  possesses a non-trivial zero if $p\geq 17$.    
For septic and unidecic forms bounds $q_0(7)\leq 883$ and $q_0(11)\leq 8053$ are due to Wooley \cite{MR2413363}.
In this paper we shall establish $q_0(5)\leq 9$.
\begin{theorem}\label{maintheorem}
Let $F(x_1,\dots,x_n)=F(\mathbf{x})$ be a quintic form with at least $n\geq 26$ variables over a $p$-adic field $\mathbb{K}$ with residue class field of cardinality $q>9$. 
Then there exists a non-zero vector $\mathbf{x}\in \mathbb{K}^n$ with $F(\mathbf{x})=0$. 
\end{theorem} 
The proof relies on a $p$-adic minimisation procedure applicable to forms of degree $d=2,3,5,7$ and $11$
 which has been 
developed by Lewis~\cite{Lewis1}, Birch and Lewis~\cite{MR0123534} and Laxton and Lewis~\cite{MR0175884}. They showed that one may assume that $F$ is reduced, that is, the resultant of the partial derivatives does not vanish and is of minimal normalised $p$-adic valuation. It then follows from a result of Leep and Yeomans that  the reduction of $F$ over the residue class field, denoted by $\theta(F)$,   is a non-degenerate form with at least $6+s$ variables, where $s$ is the maximal affine dimension of a vector space on which $\theta(F)$ vanishes. 
If $\theta(F)$  possesses a non-singular zero, it can be lifted by Hensel's Lemma to a non-trivial zero of $F$. We recall that a non-singular zero is one  which is not a simultaneous zero of the partial derivatives. \\
We shall use certain properties of quintic forms to choose a suitable  subspace and show that it contains a non-singular zero.
For $q=11,13,16,25,27,32$ this is accomplished with the help of computer calculations. 
The author was able to carry those out on his personal notebook.
This,  together with the previously mentioned results of Leep and Yeomans and Heath-Brown, yields Theorem \ref{maintheorem}.\\
There is numerical evidence to suggest that the imposed constraint on $q$ can be further reduced.
Given the current state of technology, it certainly seems doubtful to expect an answer for all $q$
at this stage.
\section{Preliminaries}
Let $\mathbb{K}$ denote a $p$-adic field with normalised valuation $\nu$, residue class field $\mathbb{F}_q$ and ring of integers $\mathcal{O}_{\mathbb{K}}$. As we are interested in a zero, we may assume from now on that $F$ has coefficients in $\mathcal{O}_{\mathbb{K}}$ and is non-degenerate.\\ 
 We call two forms $F$ and $G$ over $\mathcal{O}_{\mathbb{K}}$ equivalent if there exists a matrix $A\in GL_n(\mathbb{K})$ and $c\in \mathbb{K}^{\times}$ such that $cF(A\mathbf{x})=G(\mathbf{x})$. 
In order to state the first lemma we denote by $\mathcal{I}(F)$ the resultant of the $n$ partial derivatives of $F$. Laxton and Lewis have shown that if  $\mathcal{I}(F)=0$, then there exists a sequence of forms $F_i$ with $\mathcal{I}(F_i)\neq 0$ converging to $F$. This observation results in the following lemma. 
\begin{lemma}[{\cite[Corollary to Lemma 6]{MR0175884}}]\label{Lemma1}
In order to prove that any form of degree $d$ over a $p$-adic field $\mathbb{K}$ in $n>d^2$ variables has a non-trivial zero  it is sufficient to prove this fact for forms with $\mathcal{I}(F)\neq 0$.
\end{lemma}
We call $F$ reduced if $\mathcal{I}(F)\neq 0$ and $\nu(\mathcal{I}(F))$ is minimal among all forms equivalent to $F$.
Thus we may assume by Lemma \ref{Lemma1}  that $F$ is reduced. 
This yields suitable implications on the number of variables 
 of $\theta(F)$.
\begin{lemma}[{\cite[Proposition 4.3]{MR1382749}}] \label{lemma2}
Let $F$ be a reduced quintic form in at least $26$ variables over a $p$-adic field and $s\geq 0$ be an integer such that $\theta(F)$ vanishes on an affine $s$-dimension linear plane $V$. If $s>1$ we assume in addition that $q\geq 5$. 
We then obtain that $\theta(F)$ is a non-degenerate form in at least $6+s$ variables.
\end{lemma}
 The next lemma  shows in particular that $s\geq 1$.
Throughout this paper we shall denote by $Z(f)$ the set of projective zeros of a form $f$ over $\mathbb{F}_q$.
\begin{lemma}[Chevalley-Warning Theorem] \label{Warning} 
Let $f$ be a form  of degree $d$ over $\mathbb{F}_q$ in $n$ variables. If $n>d$ we have
\begin{align}|Z(f)|\geq \frac{{q^{n-d}-1}}{q-1}\notag.\end{align}
\end{lemma}
A proof of this classical result can be found in \cite{2847218848221}.
Lemmas \ref{lemma2} and \ref{Warning} yield the following  consequence. 
\begin{corollary} \label{Cor1} 
Let $F$ be a quintic form in at least $26$ variables over $\mathcal{O}_{\mathbb{K}}$ that does not have a non-trivial zero. Let $s$ be as defined  in Lemma \ref{lemma2}. 
We then have 
\begin{align}|Z(\theta(F))|\geq \frac{q^{s+1}-1}{q-1}\notag.\end{align}
\end{corollary}
A zero of $\theta(F)$ is not sufficient for a non-trivial zero of $F$, instead we require a non-singular zero. Once we have found one, we can apply the version of Hensel's Lemma given below.
\begin{lemma}[Hensel's Lemma]\label{sdkdkwiuujdjude}
Let $F\in \mathcal{O}_{\mathbb{K}}[x_1,\dots,x_n]$. If $\theta(F)$ has a non-singular zero, then $F$ has a non-trivial zero in $\mathbb{K}^n$.
\end{lemma}
For a discussion of Hensel's Lemma see \cite{MR0241358}, for example.


\section{Proof of Theorem \ref{maintheorem}}
Let $F$ be a quintic form in  at least $26$ variables over a $p$-adic field $\mathbb{K}$ with residue class field of cardinality $q>9$. Throughout this section we shall write $f$ for the reduction $\theta(F)$. We denote the linear span of  vectors $\mathbf{v}_1,\dots,\mathbf{v}_l\in \mathbb{F}_q^n$ by $\langle \mathbf{v}_1,\dots,\mathbf{v}_l\rangle$.\\ By Lemma \ref{Lemma1} we may assume that $F$ is reduced. It then follows by Lemma \ref{lemma2}, that $f$ is a non-degenerate form in at least $6+s$ variables, where $s$ is the maximal affine dimension of a linear subspace of $Z(f)$. \\
Suppose that $f$ does not have a non-singular zero.  
We show that there are at least four linearly independent zeros \begin{align}\mathbf{z}_1, \mathbf{z}_2, \mathbf{z}_3, \mathbf{z}_4\in Z(f)\text{ such that } \langle \mathbf{z}_i,\mathbf{z}_j \rangle  \nsubseteq Z(f)  \notag \end{align} for all $1\leq i<j\leq 4$. 
 Hence the form 
\begin{align}g(x_1,x_2,x_3,x_4):=f(x_1\mathbf{z}_1+x_2\mathbf{z}_2+x_3\mathbf{z}_3+x_4\mathbf{z}_4)\notag \end{align} 
must be of a certain shape. In particular, certain coefficients of $g$  do not vanish. We then prove the existence of a non-singular zero of $g$, contrary to our assumption. This is achieved  by considering successively larger subspaces of $\langle \mathbf{z}_1,\mathbf{z}_2,\mathbf{z}_3,\mathbf{z}_4 \rangle $ and sieving out forms  possessing non-singular zeros.\\
As a first step, we prove that there are five distinct non-zero vectors \begin{align}\mathbf{z}_1, \dots, \mathbf{z}_5 \in Z(f) \notag\end{align} such that $\mathbf{z}_1$, $\mathbf{z}_2$, $\mathbf{z}_3$ are linearly independent and $f$ does not vanish on any plane spanned by two vectors of one of the quadruples \begin{align}\{\mathbf{z}_1,\mathbf{z}_2,\mathbf{z}_3,\mathbf{z}_i\}\quad \text{where $i=4,5$.}\notag\end{align} 
In order to establish this, we begin by showing that there are three distinct subspaces
$V_1, V_2, V_3\subseteq Z(f)$ 
 of maximal dimension and two zeros $\mathbf{z}_1$, $\mathbf{z}_2\in Z(f)$ such that  \begin{align}\mathbf{z}_1,\mathbf{z}_2\notin \bigcup_{i=1}^3V_i \quad \text{ and $\quad\langle \mathbf{z}_1$, $\mathbf{z}_2 \rangle \nsubseteq Z(f)$}. \notag\end{align} Secondly, we prove the existence of a third zero  $\mathbf{z}_3\in V_3\backslash(V_1\cup V_2)$ such that $\mathbf{z}_1$, $\mathbf{z}_2$, $\mathbf{z}_3$ are linearly independent. Thirdly, we show that there is a fourth zero  $\mathbf{z}_4\in V_2\backslash V_1$ completing the first quadruple and finally, we will choose a fifth zero $\mathbf{z}_5 \in V_1$ completing the second quadruple. \\
For convenience, we first state a basic lemma and give the details of the argument outlined afterwards.
\begin{lemma}[{\cite[Lemma 5.1]{MR1382749}}]\label{Lemma5} Let $f$ be a quintic form over $\mathbb{F}_q$  possessing two distinct non-trivial zeros $\mathbf{z}_1$ and $\mathbf{z}_2$.
Then $f$ either has a non-singular zero or 
 \begin{align}f(x_1\mathbf{z}_1+x_2\mathbf{z}_2)=c_{12}x_1^3x_2^2+c_{21}x_2^3x_1^2\notag \end{align} 
and   $c_{12}c_{21}=0$.
If, in addition,  $|\langle \mathbf{z}_1, \mathbf{z}_2 \rangle \cap Z(f)|\geq 3$, then $f(x_1\mathbf{z}_1+x_2\mathbf{z}_2)$ either  possesses a non-singular zero or is the zero polynomial. 
\end{lemma}
\begin{proof}
We write 
\begin{align}
f(x_1\mathbf{z}_1+x_2\mathbf{z}_2)= a_1x_1^5+b_{12}x_1^4x_2+c_{12}x_1^3x_2^2+c_{21}x_2^3x_1^2+b_{21}x_2^4x_1+a_2x_2^5. \notag 
\end{align}
We may assume that $\mathbf{z}_1$ and $\mathbf{z}_2$ are singular zeros and hence 
\begin{align}f(x_1\mathbf{z}_1+x_2\mathbf{z}_2)=(c_{12}x_1+c_{21}x_2)x_1^2x_2^2\notag.\end{align}
If $c_{12}c_{21}\neq 0$ then $(-c_{21},c_{12})$ is a non-singular zero and otherwise $\langle \mathbf{z}_1, \mathbf{z}_2 \rangle \cap Z(f)=\{\mathbf{z}_1,\mathbf{z}_2\}$  or $\langle \mathbf{z}_1, \mathbf{z}_2 \rangle\subseteq Z(f)$.
\end{proof}
Since $f$ has at least $6$ variables, Lemma \ref{Warning} yields a non-trivial zero and thus we may assume $s\geq 1$.
By Corollary \ref{Cor1} we have 
\begin{align}\label{equation1}|Z(f)|> \frac{4(q^{s}-1)}{q-1},\end{align} provided $q\geq 4$.
 Thus we can  pick four distinct subspaces \begin{align}V_1,V_2,V_3,V_4\subseteq Z(f)\notag\end{align} such that $V_i$ is of maximal dimension for $1\leq i \leq 4$.  By equation (\ref{equation1}) 
we can choose an additional
zero $\mathbf{z}_1\in Z(f)\backslash\bigcup_{i=1}^4V_i$. We set $S_3:=\bigcup_{i=1}^3 V_i$ and 
 show that there exists a vector $\mathbf{z}_2\in V_4\backslash S_3$ such that $\langle \mathbf{z}_1,\mathbf{z}_2  \rangle \nsubseteq Z(f)$. Suppose by the contrary that \begin{align}\text{for all $\mathbf{z}\in V_4\backslash S_3$ we have $ \langle \mathbf{z}_1,\mathbf{z} \rangle \subseteq Z(f)$.}\label{equation2}\end{align}
If $V_4\cap S_3=\{0\}$, then (\ref{equation2}) contradicts the maximality of $V_4$ and otherwise we shall argue as follows. Let $\mathbf{s}\in V_4\cap S_3$ be arbitrary. As $V_4$ is distinct from $S_3$ we can choose a non-zero vector $\mathbf{v}\in V_4\backslash S_3$ and consider the  projective line $L_{\mathbf{s}}:=\langle \mathbf{v},\mathbf{s} \rangle$. Since $\mathbf{v}\notin S_3$, the projective line $L_{\mathbf{s}}$ can not contain two vectors of $V_i$ for each $1\leq i\leq 3$.  Thus the intersection $L_{\mathbf{s}} \cap S_3$ contains at most three non-zero points.  On the other hand, since $q\geq 5$, there are at least three points $\mathbf{p}_1,\mathbf{p}_2,\mathbf{p}_3 \in  L_{\mathbf{s}}$ not contained in $S_3$. It follows from our assumption (\ref{equation2}) that $\langle \mathbf{z}_1,\mathbf{p}_i \rangle\subseteq Z(f)$ for all $1\leq i \leq 3$. 
\begin{lemma}\label{Lemma6} Let $f$ be a quintic form over $\mathbb{F}_q$ without a non-singular zero, $L$ a projective line,  $\mathbf{z}$  a non-zero point not on $L$ and $\mathbf{p}_1,\mathbf{p}_2,\mathbf{p}_3\in L$ three distinct non-zero points.  Assume that \begin{align}\langle \mathbf{p}_i,\mathbf{z}  \rangle \subseteq Z(f) \quad \text{for all $1\leq i \leq 3$}\notag.\end{align}  
Then $\langle L,\mathbf{z} \rangle \subseteq Z(f)$. 
\end{lemma}
\begin{proof}
Let $\mathbf{x}\in \langle L, \mathbf{z} \rangle$ and $\mathbf{x}\notin \bigcup_{i=1}^3 \langle \mathbf{p}_i,\mathbf{z} \rangle$. 
 There exists a projective line $H$ in $\langle L, \mathbf{z} \rangle$ through $ \mathbf{x}  $ that does not contain $ \mathbf{z} $. Since we have assumed that $\mathbf{x}\notin  \langle \mathbf{p}_i,\mathbf{z} \rangle$ and $\langle \mathbf{p}_i,\mathbf{z} \rangle$ has co-dimension $1$ in $\langle L,\mathbf{z} \rangle$, the line $H$ intersects $\langle \mathbf{p}_i,\mathbf{z} \rangle$ in exactly one point $\mathbf{s}_i$, say, for each $1\leq i\leq 3$. Since $\bigcap_{i=1}^3\langle \mathbf{p}_i,\mathbf{z} \rangle=  \mathbf{z} $ and $\mathbf{z} \notin H$, we conclude that there are at least three distinct points, namely  $\mathbf{s}_i$ for $1\leq i \leq 3$, in $H$ that are contained in $Z(f)$. By Lemma \ref{Lemma5}  we have $H\subseteq Z(f)$ and hence $\mathbf{x}\in Z(f)$. We conclude that $\langle L, \mathbf{z} \rangle \subseteq Z(f)$.
\end{proof}
By applying Lemma \ref{Lemma6} we have $\langle \mathbf{z}_1,V_4 \rangle\subseteq Z(f)$,  contrary to the maximality of the dimension of $V_4$.
We conclude that there are three non-identical subspaces $V_1,V_2,V_3\subseteq Z(f)$ of maximal dimension and two zeros $\mathbf{z}_1,\mathbf{z}_2 \notin \bigcup_{i=1}^3 V_i$ such that \begin{align}\langle \mathbf{z}_1,\mathbf{z}_2 \rangle \cap Z(f)=\{\mathbf{z}_1,\mathbf{z}_2\}\notag. \end{align}
As mentioned above we shall proceed by proving the existence of a third vector $\mathbf{z}_3 \in V_3\backslash (V_1\cup V_2)$ with the property $\langle \mathbf{z}_i,\mathbf{z}_j \rangle \nsubseteq Z(f)$ for all $1\leq i<j\leq 3$. Suppose by  the contrary that for every $\mathbf{z} \in V_3\backslash (V_1\cup V_2)$ at least one of the following holds \begin{align}\label{equation3}\langle  \mathbf{z},\mathbf{z}_1 \rangle \subseteq Z(f)\quad \text{or} \quad \langle \mathbf{z},\mathbf{z}_2 \rangle \subseteq Z(f).\end{align}
We set $S_2:=V_1\cup V_2$ for shorter notation and shall argue that we may assume $S_2\cap V_3=\{0\}$. Suppose there exists at least one non-zero vector $\mathbf{s}\in S_2\cap V_3$. We then pick a vector $\mathbf{v}\in V_3\backslash  S_2$ and define 
for any vector $\mathbf{s}\in S_2\cap V_3$ 
the projective line $L_{\mathbf{s}}:=\langle \mathbf{s},\mathbf{v} \rangle$.  We show that 
\begin{align}\label{equation4}\langle L_{\mathbf{s}},\mathbf{z}_1 \rangle \subseteq Z(f)\quad \text{or} \quad \langle L_{\mathbf{s}},\mathbf{z}_2 \rangle \subseteq Z(f). \end{align}
Since $\mathbf{v}\notin S_2$, neither two vectors of the subspace $V_1$ nor two of the subspace $V_2$ can be contained in $L_{\mathbf{s}}$. 
Thus there are at least $5$ projective points in $L_{\mathbf{s}}\backslash  S_2$, provided $q\geq 6$. By our assumption (\ref{equation3}) there are three points $\mathbf{p}_1,\mathbf{p}_2,\mathbf{p}_3$ among them such that $\langle \mathbf{p}_i,\mathbf{z}_k \rangle \subseteq Z(f)$ for all $1\leq i \leq 3$ and  a certain $1\leq k\leq 2$. 
Equation (\ref{equation4}) then follows from  Lemma \ref{Lemma6} and thus, we have that for every $\mathbf{z}\in V_3$ at least one of the following holds \begin{align}\label{equation5}\langle \mathbf{z},\mathbf{z}_1 \rangle \subseteq Z(f)\quad \text{or} \quad \langle \mathbf{z},\mathbf{z}_2 \rangle \subseteq Z(f). \end{align}
\begin{lemma} \label{Lemma7}  
Let $f$ be a quintic form over $\mathbb{F}_q$ without a non-singular zero, $V\subseteq Z(f)$ an $m$-dimensional subspace where $m\geq 2$ and $\mathbf{z}_1,\dots, \mathbf{z}_k$ non-trivial zeros not contained in $V$. 
We assume $q\geq 2k$ and that there exists for any projective plane $W\subseteq V$ of co-dimension $1$ an index $i\in \{1,\dots,k\}$ such that $\langle W,\mathbf{z}_i \rangle \subseteq Z(f)$. Then there exists an index $i\in \{1,\dots,k\}$ such that \begin{align}\langle V,\mathbf{z}_i \rangle \subseteq Z(f). \notag \end{align}
\end{lemma}
\begin{proof}
We write $[x_1:\dots:x_m]$ for a projective point in $V$. Since $m\geq 2$ we can define the following subspaces 
\begin{align}
W_{(a,b)}:=\{[x_1:\dots : ax_{m-1}:bx_{m-1}]\mid x_i \in \mathbb{F}_q \text{  for $1 \leq i \leq m$} \} \notag 
\end{align}
for $(a,b)\in (\{1\}\times \mathbb{F}_q)\cup \{(0,1)\}$.\\ 
Since $q\geq 2k$ there are at least $2k+1$ subspaces $W_{(a,b)}$. Thus we may assume that there are at least three subspaces, $W_1$, $W_2$, $W_3$ say, among these and a zero $\mathbf{z}\in\{\mathbf{z}_1,\dots, \mathbf{z}_k\}$ such that
\begin{align}
 \langle W_i, \mathbf{z} \rangle \subseteq Z(f) \quad \text{ for $1\leq i \leq 3$.} \notag
\end{align}  
We shall complete the proof of this lemma by following Leep and Yeomans [\cite{MR1382749}, Lemma 5.3]. For $W_1,W_2,W_3$ as above, we have 
\begin{align} 
\langle W_i,\mathbf{z} \rangle \cap \langle W_j,\mathbf{z} \rangle &= \langle W_i\cap W_j, \mathbf{z}\rangle, \label{equation6}\\
\langle W_i, \mathbf{z}\rangle \cap \langle W_j,\mathbf{z} \rangle &= \bigcap_{i=1}^3 \langle W_i, \mathbf{z} \rangle \label{equation6b}
\end{align}
for any $1\leq i<j\leq 3$. 
We notice that for equation (\ref{equation6}) we have for each pair $i\neq j$ with $\langle W_i,\mathbf{z} \rangle$ and $\langle W_j,\mathbf{z} \rangle$ two non-identical $m$-dimensional planes and that $\langle W_i\cap W_j, \mathbf{z}\rangle$ is an $m-1$ dimensional plane. Equation (\ref{equation6b}) follows from (\ref{equation6})	 and the fact that
 \begin{align}W_i\cap W_j=\bigcap_{i=1}^3 W_i \quad\text{for distinct $i,j$.} \notag\end{align}\\ 
Let $\mathbf{x}$ be a point in $\langle V,\mathbf{z} \rangle \backslash  \bigcup_{i=1}^3  \langle W_i,\mathbf{z} \rangle$. We observe that $\bigcap_{i=1}^3 W_i$ has co-dimension $2$ in $ V $. Thus, we conclude by (\ref{equation6}) and (\ref{equation6b}) that $\bigcap_{i=1}^3 \langle W_i,\mathbf{z} \rangle$ has co-dimension $2$ in $\langle V,\mathbf{z} \rangle$. Hence we can choose a projective line $H$ through the point $\mathbf{x}$ that does not intersect with $\bigcap_{i=1}^3\langle W_i,\mathbf{z} \rangle $. Since $\mathbf{x}\notin \langle W_i,\mathbf{z} \rangle$  and $\langle W_i,\mathbf{z} \rangle$ has co-dimension $1$ in $\langle V,\mathbf{z} \rangle$, we conclude that there exists for each $i$ a point $\mathbf{p}_i\in \langle W_i,\mathbf{z} \rangle \cap H$. Since $ \langle W_i,\mathbf{z}\rangle \subseteq Z(f)$ and $H$ does not intersect $\bigcap_{i=1}^3\langle W_i,\mathbf{z} \rangle$ there are at least three distinct non-trivial zeros of $f$ on $H$. Thus we conclude by Lemma \ref{Lemma5} that $\langle V,\mathbf{z} \rangle \subseteq Z(f)$.
\end{proof}
We apply Lemma \ref{Lemma7} to (\ref{equation5}) and thus, we have  \begin{align}\langle V_3,\mathbf{z}_1 \rangle \subseteq Z(f)\quad \text{or} \quad \langle V_3,\mathbf{z}_2 \rangle \subseteq Z(f). \notag\end{align} However, this contradicts the maximality of the dimension of $V_3$. Moreover, the vectors $\mathbf{z}_1,\mathbf{z}_2,\mathbf{z}_3$ are linearly independent, since   by Lemma \ref{Lemma5} there are at most two zeros on the projective line $\langle \mathbf{z}_1,\mathbf{z}_2 \rangle $. Thus we have found three linearly independent vectors $\mathbf{z}_1,\mathbf{z}_2,\mathbf{z}_3$ such that \begin{align}\langle \mathbf{z}_i,\mathbf{z}_j \rangle \nsubseteq Z(f) \quad \text{for all $1\leq i<j\leq 3$}.\notag\end{align}
We show that there exists a fourth vector $\mathbf{z}_4 \in V_2\backslash V_1$ such that \begin{align}\langle \mathbf{z}_i,\mathbf{z}_j \rangle \nsubseteq Z(f) \quad \text{ for all $1\leq i<j\leq 4$.}\notag \end{align} Suppose by  the contrary that for all $\mathbf{z} \in V_2\backslash V_1$ at least one of the following holds \begin{align}\label{equation7}\langle \mathbf{z},\mathbf{z}_1 \rangle \subseteq Z(f),\quad  \langle \mathbf{z},\mathbf{z}_2 \rangle \subseteq Z(f)\quad \text{or} \quad  \langle \mathbf{z},\mathbf{z}_3 \rangle \subseteq Z(f).\end{align}
We shall argue that there is no loss of generality if we assume $V_1\cap V_2=\{0\}$. 
 As there exists a point $\mathbf{v}\in V_2\backslash V_1$ we
consider for any vector $\mathbf{s}\in  V_2\cap V_1$ the plane  $L_{\mathbf{s}}:=\langle \mathbf{s},\mathbf{v} \rangle $. We show that 
 \begin{align}\langle L_{\mathbf{s}},\mathbf{z}_1 \rangle \subseteq Z(f),\quad \langle L_{\mathbf{s}},\mathbf{z}_2 \rangle \subseteq Z(f)\quad \text{or} \quad \langle L_{\mathbf{s}},\mathbf{z}_3 \rangle \subseteq Z(f). \notag  \end{align}
Since $q\geq 7$ there are at least $7$ projective points in $L_{\mathbf{s}}$ not contained in $V_1$. Thus, by (\ref{equation7}) there are  three points $\mathbf{p}_1,\mathbf{p}_2,\mathbf{p}_3$ among them such that $\langle \mathbf{p}_i,\mathbf{z}_k \rangle \subseteq Z(f)  $ for all $1\leq i \leq 3$ and  a certain $1\leq k\leq 3$.
By Lemma \ref{Lemma6}, we have that for every $\mathbf{z}\in V_2$ at least one of the following holds \begin{align} \label{equation8}\langle \mathbf{z},\mathbf{z}_1 \rangle \subseteq Z(f), \quad \langle \mathbf{z},\mathbf{z}_2 \rangle \subseteq Z(f) \quad \text{or} \quad  \langle \mathbf{z},\mathbf{z}_3 \rangle \subseteq Z(f). \end{align}
It then follows in conjunction with Lemma \ref{Lemma7} that \begin{align}\langle V_2,\mathbf{z}_1 \rangle \subseteq Z(f),\quad \langle V_2,\mathbf{z}_2 \rangle \subseteq Z(f)\quad \text{or} \quad  \langle V_2,\mathbf{z}_3 \rangle\subseteq Z(f). \notag\end{align} However, any of those contradicts the maximality of the dimension of $V_2$ and hence we may assume the existence of a vector $\mathbf{z}_4 \in V_2\backslash V_1$ such that \begin{align}\langle \mathbf{z}_i,\mathbf{z}_j \rangle \nsubseteq Z(f)\quad \text{for all $1\leq i<j\leq 4$}.\notag\end{align}
We show that there exists a fifth vector $\mathbf{z}_5 \in V_1$ such that \begin{align}\langle \mathbf{z}_i,\mathbf{z}_5 \rangle \nsubseteq Z(f)\notag \quad \text{for all $1\leq i\leq 3$.}\end{align} Suppose by  the contrary that for all $\mathbf{z} \in V_1$ at least one of the conditions in equation (\ref{equation7}) holds. 
By Lemma \ref{Lemma7} this implies  \begin{align}\langle V_1,\mathbf{z}_1 \rangle \subseteq Z(f),\quad  \langle  V_1,\mathbf{z}_2 \rangle \subseteq Z(f) \quad \text{or} \quad  \langle  V_1,\mathbf{z}_3 \rangle \subseteq Z(f).\notag\end{align}
However, any of these contradicts the maximality of the dimension of $V_1$ and thus we conclude that there is a vector $\mathbf{z}_5 \in V_1$ such that \begin{align}\langle \mathbf{z}_i,\mathbf{z}_5 \rangle \nsubseteq Z(f)\quad \text{for all $1\leq i\leq 3$.}\notag\end{align}
In summary,  we have shown that there are two quadruples of zeros, \begin{align}\text{$\mathbf{z}_1$, $\mathbf{z}_2$, $\mathbf{z}_3$, $\mathbf{z}_4\quad$ and $\quad\mathbf{z}_1$, $\mathbf{z}_2$, $\mathbf{z}_3$, $\mathbf{z}_5$,}\notag\end{align} such that $f$ does not vanish on any two-dimensional plane spanned by two zeros of one quadruple. Moreover, we know that $\mathbf{z}_1,\mathbf{z}_2,\mathbf{z}_3$ are linearly independent. We will now estimate the number of zeros of $f$ in $\langle \mathbf{z}_1,\mathbf{z}_2,\mathbf{z}_3 \rangle$. 
\begin{lemma} \label{Lemma8}
Let $f$ be a quintic form over $\mathbb{F}_q$ with three linearly independent zeros $\mathbf{z}_1,\mathbf{z}_2,\mathbf{z}_3\in Z(f)$ such that  $\langle \mathbf{z}_i,\mathbf{z}_j \rangle \nsubseteq Z(f) $ for all $1\leq i<j\leq 3$.
 Then the following holds.\\ If $q\geq 17$, then $f$ has a non-singular zero.
 If $ 11\leq  q<17$, it possesses a non-singular zero or 
$|\langle \mathbf{z}_1,\mathbf{z}_2,\mathbf{z}_3 \rangle\cap Z(f)|= 3$ holds. 
  If $q<11$ it has a non-singular zero or 
	$|\langle \mathbf{z}_1,\mathbf{z}_2,\mathbf{z}_3 \rangle\cap Z(f)|\leq  4$ holds.
\end{lemma}
The last inequality is sharp. For instance, 
\begin{align} 
2x_1^3x_2^2+2x_1^3x_3^2+4x_2^3x_3^2+x_1x_2x_3(5x_1^2+6x_2^2+2x_3^2+x_1x_2+x_1x_3+x_2x_3) \notag 
\end{align} 
is a form over $\mathbb{F}_7$ 
 possessing exactly four  singular zeros, namely \begin{align}\langle (1,0,0) \rangle, \langle (0,1,0) \rangle, \langle (0,0,1) \rangle,  \langle (1,6,2) \rangle.\notag\end{align}
\begin{proof}
Suppose that $f$ does not have a non-singular zero. Thus we can write $f(x_1\mathbf{z}_1+x_2\mathbf{z}_2+x_3\mathbf{z}_3)$ as
\begin{align}
x_1x_2x_3Q(x_1,x_2,x_3)+\sum_{\substack{1\leq i< j \leq 3}}c_{  ij }x_i^3x_j^2+c_{  ji }x_j^3x_i^2 \notag 
\notag 
\end{align}
where $Q(x_1,x_2,x_3)$ is a quadratic form.   
By applying Lemma \ref{Lemma5} to any two variables of $f(x_1\mathbf{z}_1+x_2\mathbf{z}_2+x_3\mathbf{z}_3)$ we have  
$c_{ij}c_{ji}=0$ for all $1\leq i<j\leq 3$. Since $f$ does not vanish on any of the projective lines $\langle \mathbf{z}_i,\mathbf{z}_j \rangle$ with $1\leq i<j\leq 3$, we have either \begin{align}\text{$c_{ij}\neq 0\quad$ or $\quad c_{ji}\neq 0\quad$ for all $1\leq i<j\leq 3$.}\notag\end{align}
Hence, we see after permuting the variables that $f(x_1\mathbf{z}_1+x_2\mathbf{z}_2+x_3\mathbf{z}_3)$ takes one of the following shapes 
\begin{align}
&t_1(x_1,x_2,x_3)=c_{12}x_1^3x_2^2+c_{13}x_1^3x_3^2+c_{23}x_2^3x_3^2+x_1x_2x_3Q(x_1,x_2,x_3), \notag\\
&t_2(x_1,x_2,x_3)=c_{12}x_1^3x_2^2+c_{31}x_3^3x_1^2+c_{23}x_2^3x_3^2+x_1x_2x_3Q(x_1,x_2,x_3),\notag
\end{align}
where $Q(x_1,x_2,x_3)$ is a quadratic form and $c_{12},c_{13},c_{23}$ and $c_{31}$ are all non-zero coefficients. \\
It has been proved by Leep and Yeomans \cite{MR1382749} using the Lang-Weil Bound
 that $f(x_1\mathbf{z}_1+x_2\mathbf{z}_2+x_3\mathbf{z}_3)$ has always a non-singular zero, provided $q\geq  43$. Heath-Brown \cite{MR2595750} has extended this to prime values of $q\geq 17$.\\
Similarly, we show  by computer calculations that $f$ has a non-singular zero for $q=25,27,32$. In each case there are,  after an appropriate rescaling of both, the forms $t_1$, $t_2$ and the variables, just $6$ degrees of freedom. A computer program  can verify the existence of a non-singular zero for each form $t_1$, respectively each form $t_2$, by successively testing points in $\mathbb{F}_q^3$.\\
 If $q< 17$ it can be checked by an analogous computer calculation that $t_1$ and $t_2$ either  possess a non-singular zero or that the bound on $|\langle \mathbf{z}_1,\mathbf{z}_2,\mathbf{z}_3 \rangle\cap Z(f)|$ holds.
\end{proof}
Lemma \ref{Lemma8} establishes Theorem \ref{maintheorem}, provided $q\geq 17$. Moreover, it shows that not both quadruples $\mathbf{z}_1,\mathbf{z}_2,\mathbf{z}_3,\mathbf{z}_4$ and $\mathbf{z}_1,\mathbf{z}_2,\mathbf{z}_3,\mathbf{z}_5$ can consist of linearly dependent vectors.  Thus we  may assume, after renaming, 
that we have linearly independent vectors $\mathbf{z}_1,\mathbf{z}_2,\mathbf{z}_3,\mathbf{z}_4$ such that \begin{align}\langle \mathbf{z}_i,\mathbf{z}_j  \rangle \nsubseteq Z(f)\quad \text{for all $1\leq i<j\leq 4$}.\notag\end{align} We write $f(x_1\mathbf{z}_1+x_2\mathbf{z}_2+x_3\mathbf{z}_3+x_4\mathbf{z}_4)$ as
\begin{align}\label{equation10}
\sum_{\substack{i\neq j}}a_{ij}x_i^3x_j^2 + \sum_{\substack{ k\neq i,j\\ i<j}}b_{ijk}x_ix_jx_k^3+ \sum_{\substack{i \neq j,k\\  j<k}}c_{ijk}x_ix_j^2x_k^2 + \sum_{\substack{l \neq i,j,k\\i<j<k}}d_{ijkl}x_ix_jx_kx_l^2,
\end{align}
where $1 \leq i,j,k \leq 4$.
By applying Lemma \ref{Lemma5} 
 and since $f$ does not vanish on any of the projective lines $\langle \mathbf{z}_i,\mathbf{z}_j \rangle$, we conclude  that for each pair $(i,j)$ with $i\neq j$ exactly one of $a_{ij}$ and $a_{ji}$ is zero. It then follows that, after a permutation of the variables,  the form
(\ref{equation10}) can take only four different shapes.  If we write $h$ for
\begin{align}
a_{23}x_2^3x_3^2&+a_{24}x_2^3x_4^2+a_{34}x_3^3x_4^2+\notag \\  & \sum_{\substack{ k\neq i,j\\ i<j}}b_{ijk}x_ix_jx_k^3+  \sum_{\substack{i \neq j,k\\  j<k}}c_{ijk}x_ix_j^2x_k^2 				+ \sum_{\substack{l \neq i,j,k\\  i<j<k}}d_{ijkl}x_ix_jx_kx_l^2 \notag
\end{align}
those are
\begin{align} 
&g_1:=a_{12}x_1^3x_2^2+a_{13}x_1^3x_3^2+a_{14}x_1^3x_4^2+h, \notag \\
&g_2:=a_{12}x_1^3x_2^2+a_{31}x_3^3x_1^2+a_{14}x_1^3x_4^2+h,   \notag \\
&g_3:=a_{12}x_1^3x_2^2+a_{13}x_1^3x_3^2+a_{41}x_4^3x_1^2+h,  \notag\\
&g_4:=a_{21}x_2^3x_1^2+a_{13}x_1^3x_3^2+a_{41}x_4^3x_1^2+h.  \notag 
\end{align}
As indicated it has been checked on a computer that each of those forms has a non-singular zero, provided $9< q\leq 16$. We briefly describe the assembling process.\\
 Along the way, we have already excluded,  via Lemma \ref{Lemma5}, all forms that have a non-singular zero on one of the projective lines $\langle \mathbf{z}_i,\mathbf{z}_j  \rangle$ for  some $1\leq i<j\leq 4$. Furthermore, we know from the proof of Lemma \ref{Lemma8}  all forms which do not have a non-singular zero in one of the subspaces  \begin{align}\langle \mathbf{z}_i,\mathbf{z}_j,\mathbf{z}_k  \rangle \quad \text{for some $1\leq i<j<k\leq 4$.}\notag \end{align}
Note that $g_1,g_2,g_3$ and $g_4$ restricted to such a subspace are, after permuting the variables, equal to  $t_1$ or $t_2$ as stated in the proof of Lemma \ref{Lemma8}. The computer programs for $g_1,g_2,g_3$ and $g_4$ are analogous. Suppose $g_s$ for some $1\leq s \leq 4$ is one of these cases. 
 We save  the  rearranged  coefficients of those forms of shape $t_1$, respectively $t_2$,  without a non-singular zero  in four multidimensional arrays 
\begin{align}A_{ijk}[\star,\star] \quad \text{where $1\leq i<j<k\leq 4$}\notag\end{align} such that they represent the coefficients of $g_s$ restricted to the subspace $\langle \mathbf{z}_i,\mathbf{z}_j,\mathbf{z}_k  \rangle$. 
Thus, every set of coefficients of the form $g_s|_{\langle \mathbf{z}_i,\mathbf{z}_j,\mathbf{z}_k  \rangle}$ without a non-singular zero corresponds to a line $A_{ijk}[r,\star]$.\\
 We use these data to  construct all remaining forms by combining data in these arrays and four additional degrees of freedom. Let $r_{ijk}$ denote the $r_{ijk}$-th line of $A_{ijk}[\star,\star]$ for $1\leq i<j<k \leq 4$. The  non-negative integers $r_{123},r_{124},r_{134},r_{234}$, provided the corresponding lines are compatible with respect to the coefficients they share, determine a form \begin{align}C(r_{123},r_{124},r_{134},r_{234})\notag\end{align} in four variables, $x_1,x_2,x_3,x_4$ say, with each monomial in at most three variables. Thus  any relevant form of  shape $g_s$ can be written as 
\begin{align}C(r_{123}&,r_{124},r_{134},r_{234};a,b,c,d)\notag \\&=C(r_{123},r_{124},r_{134},r_{234})+x_1x_2x_3x_4(ax_1+bx_2+cx_3+dx_4).\notag \end{align} 
For all admissible $r_{123},r_{124},r_{134},r_{234}$ and for all $a,b,c,d\in \mathbb{F}_q$ we then search for a non-singular zero $(x_1,x_2,x_3,x_4) \in \mathbb{F}_q^4$ of \begin{align}C(r_{123},r_{124},r_{134},r_{234};a,b,c,d)\notag \end{align} by  trying points successively. 
To do this efficiently, one can rescale both the forms and variables. For instance, rescale $g_1,g_2,g_3$ such that 
\begin{align}a_{12}=1,\quad a_{23}=1,\quad a_{34}=1\notag\end{align} and $g_4$ such that 
\begin{align}a_{21}=1,\quad a_{23}=1,\quad a_{34}=1.\notag\end{align} It is easier to choose a rescaling that is compatible with the one used in Lemma \ref{Lemma8} (and hence with the data in the arrays $A_{ijk}[\star,\star]$). Besides these considerations, we put a general effort on implementing the algorithm efficiently.\\
The full C++ program and the data used in the assembling process are available at \cite{2013arXiv1308.0999D}.
 This completes the proof of Theorem \ref{maintheorem}.\\

  Note that apart from the computer checks we have not used any assumption other than $q>5$.  For $q=8,9$ it is likely that one can also find by a computer search a non-singular zero of every form of the shapes $g_1,g_2,g_3$ and $g_4$. Whereas the case $q=7$ seems more doubtful than $q=8,9$, one can easily find counterexamples, for instance of shape $g_1$, for $q=5$ using the same algorithm.


\bibliographystyle{amsplain}
\bibliography{Bib}
 
\noindent\textsc{\small Mathematisches Institut, Bunsenstr. 3-5, 37073 Göttingen,}\\ \textsc{Germany}\\
\emph{ \small E-mail address:} \texttt{jdumke@uni-math.gwdg.de}

\end{document}